\documentclass[12pt,leqno]{amsart}
\usepackage{amsmath,amsthm,amsfonts,amssymb,latexsym, mathabx}

\usepackage{hyperref}

\usepackage{enumerate}
\usepackage{color}
\usepackage[dvipsnames]{xcolor}

\usepackage[shortlabels]{enumitem}

\begin{document}

\theoremstyle{plain}

\newtheorem{thm}{Theorem}[section]

\newtheorem{lem}[thm]{Lemma}
\newtheorem{Problem B}[thm]{Problem B}

\newtheorem{pro}[thm]{Proposition}
\newtheorem{conj}[thm]{Conjecture}
\newtheorem{cor}[thm]{Corollary}
\newtheorem{que}[thm]{Question}
\newtheorem{rem}[thm]{Remark}
\newtheorem{defi}[thm]{Definition}

\newtheorem*{thmA}{Theorem A}
\newtheorem*{thmB}{Theorem B}
\newtheorem*{corB}{Corollary B}
\newtheorem*{thmC}{Theorem C}
\newtheorem*{thmD}{Theorem D}
\newtheorem*{thmE}{Theorem E}
 
\newtheorem*{thmAcl}{Main Theorem$^{*}$}
\newtheorem*{thmBcl}{Theorem B$^{*}$}
\newtheorem{thml}{Theorem}
\renewcommand*{\thethml}{\Alph{thml}}   
\newtheorem{quel}[thml]{Question}

\newcommand{\dd}{\mathrm{d}}

\newcommand{\bC}{{\mathbf{C}}}
\newcommand{\bH}{{\mathbf{H}}}
\newcommand{\bN}{{\mathbf{N}}}
\newcommand{\bT}{{\mathbf{T}}}
\newcommand{\bZ}{{\mathbf{Z}}}
\newcommand{\Maxn}{\operatorname{Max_{\textbf{N}}}}
\newcommand{\Syl}{\operatorname{Syl}}
\newcommand{\Lin}{\operatorname{Lin}}
\newcommand{\U}{\mathbf{U}}
\newcommand{\R}{\mathbf{R}}
\newcommand{\dl}{\operatorname{dl}}
\newcommand{\Con}{\operatorname{Con}}
\newcommand{\sym}{\mathfrak{S}}
\newcommand{\cl}{\operatorname{cl}}
\newcommand{\Stab}{\operatorname{Stab}}
\newcommand{\Aut}{\operatorname{Aut}}
\newcommand{\Ker}{\operatorname{Ker}}
\newcommand{\InnDiag}{\operatorname{InnDiag}}
\newcommand{\fl}{\operatorname{fl}}
\newcommand{\Irr}{\operatorname{Irr}}
\newcommand{\FF}{\mathbb{F}}
\newcommand{\EE}{\mathbb{E}}
\newcommand{\POmega}{{\operatorname{P\Omega}}}
\newcommand{\normal}{\trianglelefteq}
\newcommand{\sn}{\normal\normal}
\newcommand{\Bl}{\mathrm{Bl}}
\newcommand{\NN}{\mathbb{N}}
\newcommand{\N}{\mathbf{N}}
\newcommand{\bfC}{\mathbf{C}}
\newcommand{\bfO}{\mathbf{O}}
\newcommand{\bfF}{\mathbf{F}}
\newcommand\wh[1]{\hstretch{2}{\hat{\hstretch{.5}{#1}}}}
\newcommand\wt[1]{\widetilde{#1}}
\newcommand\bg[1]{\mathbf{#1}}
\newcommand{\fS}{{\mathfrak{S}}}
\newcommand{\fA}{{\mathfrak{A}}}
\newcommand{\6}{^}

\def\GGG{{\mathcal G}}
\def\HHH{{\mathcal H}}
\def\HH{{\mathcal H}}
\def\irra#1#2{{\rm Irr}_{#1}(#2)}

\renewcommand{\labelenumi}{\upshape (\roman{enumi})}

\newcommand{\PSL}{\operatorname{PSL}}
\newcommand{\PSU}{\operatorname{PSU}}
\newcommand{\alt}{\operatorname{Alt}}

\providecommand{\V}{\mathrm{V}}
\providecommand{\E}{\mathrm{E}}
\providecommand{\ir}{\mathrm{Irm_{rv}}}
\providecommand{\Irrr}{\mathrm{Irm_{rv}}}
\providecommand{\re}{\mathrm{Re}}

\numberwithin{equation}{section}
\def\irrp#1{{\rm Irr}_{p'}(#1)}

\def\ibrrp#1{{\rm IBr}_{\Bbb R, p'}(#1)}
\def\C{{\mathbb C}}
\def\Q{{\mathbb Q}}
\def\irr#1{{\rm Irr}(#1)}
\def\irrp#1{{\rm Irr}_{p^\prime}(#1)}
\def\irrq#1{{\rm Irr}_{q^\prime}(#1)}
\def \c#1{{\cal #1}}
\def \aut#1{{\rm Aut}(#1)}
\def\cent#1#2{{\bf C}_{#1}(#2)}
\def\norm#1#2{{\bf N}_{#1}(#2)}
\def\E#1{{\bf{E}}(#1)}
\def\F#1{{\bf{F}}(#1)}

\def\zent#1{{\bf Z}(#1)}
\def\syl#1#2{{\rm Syl}_#1(#2)}
\def\normal{\triangleleft\,}
\def\oh#1#2{{\bf O}_{#1}(#2)}
\def\Oh#1#2{{\bf O}^{#1}(#2)}
\def\det#1{{\rm det}(#1)}
\def\gal#1{{\rm Gal}(#1)}
\def\ker#1{{\rm ker}(#1)}
\def\normalm#1#2{{\bf N}_{#1}(#2)}
\def\alt#1{{\rm Alt}(#1)}
\def\iitem#1{\goodbreak\par\noindent{\bf #1}}
   \def \mod#1{\, {\rm mod} \, #1 \, }
\def\sbs{\subseteq}

\def\gc{{\bf GC}}
\def\ngc{{non-{\bf GC}}}
\def\ngcs{{non-{\bf GC}$^*$}}
\newcommand{\notd}{{\!\not{|}}}

\newcommand{\Z}{\mathbf{Z}}
\newcommand{\Out}{{\mathrm {Out}}}
\newcommand{\Mult}{{\mathrm {Mult}}}
\newcommand{\Inn}{{\mathrm {Inn}}}
\newcommand{\IBR}{{\mathrm {IBr}}}
\newcommand{\IBRL}{{\mathrm {IBr}}_{\ell}}
\newcommand{\IBRP}{{\mathrm {IBr}}_{p}}
\newcommand{\cd}{\mathrm{cd}}
\newcommand{\ord}{{\mathrm {ord}}}
\def\id{\mathop{\mathrm{ id}}\nolimits}
\renewcommand{\Im}{{\mathrm {Im}}}
\newcommand{\Ind}{{\mathrm {Ind}}}
\newcommand{\diag}{{\mathrm {diag}}}
\newcommand{\soc}{{\mathrm {soc}}}
\newcommand{\End}{{\mathrm {End}}}
\newcommand{\sol}{{\mathrm {sol}}}
\newcommand{\Hom}{{\mathrm {Hom}}}
\newcommand{\Mor}{{\mathrm {Mor}}}
\newcommand{\Mat}{{\mathrm {Mat}}}
\def\rank{\mathop{\mathrm{ rank}}\nolimits}
\newcommand{\Tr}{{\mathrm {Tr}}}
\newcommand{\tr}{{\mathrm {tr}}}
\newcommand{\Gal}{{\rm Gal}}
\newcommand{\Spec}{{\mathrm {Spec}}}
\newcommand{\ad}{{\mathrm {ad}}}
\newcommand{\Sym}{{\mathrm {Sym}}}
\newcommand{\Char}{{\mathrm {Char}}}
\newcommand{\pr}{{\mathrm {pr}}}
\newcommand{\rad}{{\mathrm {rad}}}
\newcommand{\abel}{{\mathrm {abel}}}
\newcommand{\PGL}{{\mathrm {PGL}}}
\newcommand{\PCSp}{{\mathrm {PCSp}}}
\newcommand{\PGU}{{\mathrm {PGU}}}
\newcommand{\codim}{{\mathrm {codim}}}
\newcommand{\ind}{{\mathrm {ind}}}
\newcommand{\Res}{{\mathrm {Res}}}
\newcommand{\Lie}{{\mathrm {Lie}}}
\newcommand{\Ext}{{\mathrm {Ext}}}
\newcommand{\Alt}{{\mathrm {Alt}}}
\newcommand{\AAA}{{\sf A}}
\newcommand{\SSS}{{\sf S}}
\newcommand{\SL}{{\mathrm {SL}}}
\newcommand{\Sp}{{\mathrm {Sp}}}
\newcommand{\PSp}{{\mathrm {PSp}}}
\newcommand{\SU}{{\mathrm {SU}}}
\newcommand{\GL}{{\mathrm {GL}}}
\newcommand{\GU}{{\mathrm {GU}}}
\newcommand{\Spin}{{\mathrm {Spin}}}
\newcommand{\CC}{{\mathbb C}}
\newcommand{\CB}{{\mathbf C}}
\newcommand{\RR}{{\mathbb R}}
\newcommand{\QQ}{{\mathbb Q}}
\newcommand{\ZZ}{{\mathbb Z}}
\newcommand{\bfN}{{\mathbf N}}
\newcommand{\bfZ}{{\mathbf Z}}
\newcommand{\PP}{{\mathbb P}}
\newcommand{\cG}{{\mathcal G}}
\newcommand{\cH}{{\mathcal H}}
\newcommand{\cQ}{{\mathcal Q}}
\newcommand{\GA}{{\mathfrak G}}
\newcommand{\cT}{{\mathcal T}}
\newcommand{\cL}{{\mathcal L}}
\newcommand{\IBr}{\mathrm{IBr}}
\newcommand{\cS}{{\mathcal S}}
\newcommand{\cR}{{\mathcal R}}
\newcommand{\GCD}{\GC^{*}}
\newcommand{\TCD}{\TC^{*}}
\newcommand{\FD}{F^{*}}
\newcommand{\GD}{G^{*}}
\newcommand{\HD}{H^{*}}
\newcommand{\GCF}{\GC^{F}}
\newcommand{\TCF}{\TC^{F}}
\newcommand{\PCF}{\PC^{F}}
\newcommand{\GCDF}{(\GC^{*})^{F^{*}}}
\newcommand{\RGTT}{R^{\GC}_{\TC}(\theta)}
\newcommand{\RGTA}{R^{\GC}_{\TC}(1)}
\newcommand{\Om}{\Omega}
\newcommand{\eps}{\epsilon}
\newcommand{\varep}{\varepsilon}
\newcommand{\al}{\alpha}
\newcommand{\chis}{\chi_{s}}
\newcommand{\sigmad}{\sigma^{*}}
\newcommand{\PA}{\boldsymbol{\alpha}}
\newcommand{\gam}{\gamma}
\newcommand{\lam}{\lambda}
\newcommand{\la}{\langle}
\newcommand{\genf}{F^*}
\newcommand{\ra}{\rangle}
\newcommand{\hs}{\hat{s}}
\newcommand{\htt}{\hat{t}}
\newcommand{\tG}{\hat G}
\newcommand{\St}{\mathsf {St}}
\newcommand{\bfs}{\boldsymbol{s}}
\newcommand{\bfl}{\boldsymbol{\lambda}}
\newcommand{\tn}{\hspace{0.5mm}^{t}\hspace*{-0.2mm}}
\newcommand{\ta}{\hspace{0.5mm}^{2}\hspace*{-0.2mm}}
\newcommand{\tb}{\hspace{0.5mm}^{3}\hspace*{-0.2mm}}
\def\skipa{\vspace{-1.5mm} & \vspace{-1.5mm} & \vspace{-1.5mm}\\}
\newcommand{\tw}[1]{{}^#1\!}
\renewcommand{\mod}{\bmod \,}
\newcommand{\type}[1]{\operatorname{#1}}
\newcommand{\GC}{{\mathcal{G}}}
\newcommand{\EC}{{\mathcal{E}}}
\newcommand{\LC}{{\mathcal{L}}}
\newcommand{\HC}{{\mathcal{H}}}
\newcommand{\TC}{{\mathcal{T}}}
\newcommand{\sct}{{\mathrm {sc}}}

\newcommand{\mandicomment}{\textcolor{blue}}
\newcommand{\eugecomment}{\textcolor{red}}

\marginparsep-0.5cm

\renewcommand{\thefootnote}{\fnsymbol{footnote}}
\footnotesep6.5pt

\title{Character degrees in blocks and defect groups}

\author{Eugenio Giannelli}
\address[E. Giannelli]{Dipartimento di Matematica e Informatica U.~Dini, Viale Morgagni 67/a, Firenze, Italy}
\email{eugenio.giannelli@unifi.it}
\author{J. Miquel Mart\'inez}
\address[J. M. Mart\'inez]{Departament de Matem\`atiques, Universitat de Val\`encia,  46100 Burjassot, Val\`encia, Spain}
\email{josep.m.martinez@uv.es}
\author{A. A. Schaeffer Fry}
\address[A. A. Schaeffer Fry]{Deptartment of Mathematics and Statistics, Metropolitan State University of Denver, Denver, CO 80217, USA}
\email{aschaef6@msudenver.edu}

\thanks{The second author acknowledges support from
Spanish Ministerio de Ciencia
e Innovaci\'on
PID2019-103854GB-I00 and FEDER funds, as well as a predoctoral grant from the ``Atracci\'o del talent" programme from the Universitat de Val\`encia. The third author is partially supported by a grant from the National Science Foundation, Award No. DMS-2100912.}

\maketitle

\begin{abstract}
A recent question of Gabriel Navarro asks whether it is true that the derived length of a defect group is less than or equal to the number of degrees of irreducible characters in a block. In this article, we bring new evidence  towards the validity of this statement. 
\end{abstract}

\section{Introduction}
 One of the main themes in  the representation theory of finite groups is to relate global and local invariants. 
 For example, for a prime $p$, Brauer's Problem 12 asks what can be said about the Sylow $p$-subgroups of $G$ from the makeup of the set of irreducible characters $\Irr(G)$.  More generally, the problem of determining structural properties of a defect group of a $p$-block $B$ of $G$ based on the set $\Irr(B)$ of irreducible characters of $B$ has been an important topic of study in the area.
The proof of one direction of the Brauer Height Zero Conjecture in \cite{kessarmalle} and the recent proof for principal blocks \cite{MN21} have been major breakthroughs in this line of investigation. Other recent results studying the $p$-structural properties
of the defect groups in terms of properties of the irreducible characters in the block
can be found in \cite{FLLMZ}, \cite{GRSS} and \cite{nrsv}. The present paper is a contribution to this lively area of research. 

\medskip

Let $B$ be a Brauer $p$-block of a finite group $G$, with defect group $D$. We write
 $\cd(B)= \{\chi(1)\mid\chi\in \Irr(B)\}$ for the set of the degrees
of the complex irreducible characters in $B$ and $\mathrm{dl}(D)$ for the derived length of the solvable group $D$. This article's main motivation is the following question, recently appeared in \cite{mar21}.
\begin{quel}[Navarro]{\label{a}}
Is it true that  $\mathrm{dl}(D)\leq |\cd(B)|$? 
\end{quel}

Question \ref{a} has a positive answer 
if $|\cd(B)|=1$. In fact, a theorem of Okuyama and Tsushima \cite{OT83} shows that in this case the defect group $D$ is abelian. Similarly, Navarro's prediction is confirmed whenever $D$ is normal in $G$ by \cite[Theorem 8]{reynolds}. 
On the other hand, at the time of this writing, 
the case where $|\cd(B)|=2$ remains an open problem, although it would follow from an affirmative answer to the Brauer Height Zero Conjecture, along with a conjecture of Malle and Navarro \cite{MN11}.  
Specializing the question to the principal block $B=B_0(G)$, it was recently shown by the second author in \cite{mar21} that it is true that if $|\cd(B_0(G))|=2$ then the Sylow $p$-subgroups of $G$ have derived length at most $2$.  
In this article we extend the above mentioned result by proving that the following holds.


\begin{thml}\label{main}
Let $p$
be a prime, let $G$ be a finite group with Sylow $p$-subgroup $P$, and let $B_0(G)$ be the principal $p$-block of $G$.
If  $|\cd(B_0(G))|\leq 3$, then $\mathrm{dl}(P)\leq |\cd(B_0(G))|$.
\end{thml}

As mentioned before, Theorem B extends the results of \cite{mar21} by confirming Question \ref{a} for principal blocks with at most three distinct character degrees. 
The strategy we use to prove this statement is to reduce it to a few technical statements about the principal blocks of non-abelian simple groups (see Section \ref{sec: reduction}). 
Checking these statements led us to the discovery of some interesting features of several important classes of finite groups. For instance, we show that Question \ref{a} holds for symmetric and alternating groups, as well as for general linear groups in defining characteristic (respectively denoted by $\fS_n$, $\fA_n$ and $\GL_n(q)$) as a consequence of the following stronger statement.



\begin{thml}\label{thm: C}
Let $p$ be a prime and let $B$ be a $p$-block of $\fS_n$, $\fA_n$ or $\GL_n(q)$ with $q$ a power of $p$. Let $D$ be a defect group of $B$. Then $\mathrm{dl}(D)\leq |\mathrm{ht}(B)|.$
\end{thml}
\noindent Here $\mathrm{ht}(B)$ denotes the set of heights of the irreducible characters in $B$. 

We care to remark that it is not possible in general to bound the derived length of the defect groups by the number of distinct heights in the corresponding block. In fact, in Section \ref{sec: 5} we construct a solvable group that does not satisfy the conclusion of Theorem \ref{thm: C}.
It is worth pointing out that this construction also provides a counterexample to the main conjecture of the recent article \cite{flz}.


\subsection*{Acknowledgments}
The second author wishes to thank G. Navarro, N. Rizo and L. Sanus for helpful conversations on the topic of this paper. The third author would like to thank M. Cabanes for clarification about the principal block in defining characteristic.  The authors also thank G. Malle for his comments on an earlier version of this manuscript.


\section{Reduction to simple groups}\label{sec: reduction}

\subsection{Auxiliary lemmas}

\begin{lem}\label{gallagher-rizo} 
Let $G$ be a finite group, $N\normal G$, and $\theta\in \Irr(B_0(N))$. Assume $\theta$ extends to $\chi\in\Irr(B_0(G))$. Then $\{\mu\chi\mid\mu\in\Irr(B_0(G/N))\}\sbs\Irr(G|\theta)\cap\Irr(B_0(G))$.
\end{lem}
\begin{proof}
By Gallagher's correspondence, $\Irr(G|\theta)=\{\mu\chi\mid\mu\in\Irr(G/N)\}$. We have $\chi=1_{G/N}\chi$ and then by \cite[Lemma 2.4]{Riz18}, if $\mu\in\Irr(B_0(G/N))$ then $\mu\chi\in\Irr(B_0(G))$.\end{proof}

\begin{lem}\label{induction}
Let $G$ be a finite group and let $S_1\times\dots\times S_t\normal G$, where the $S_i$'s are subgroups permuted by $G$. Let $S=S_1$. Assume there exists $\alpha\in\Irr(\norm G S/\cent G S)$ in the principal block and such that $S$ is not contained in $\ker\alpha$. Then $\alpha^G\in \Irr(B_0(G))$.
\end{lem}
\begin{proof}
Write $H=\norm G S$ and $C=\cent G S$. Let $\eta\in\Irr(S)$ be under $\alpha$. Now let $\psi=\eta\times 1_{S_2}\times\dots\times 1_{S_t}\in \Irr(S_1\times\dots\times S_t)$. Notice that $G_\psi=H_\eta$. Since $S_2\cdots S_t\sbs C\sbs \ker\alpha$ we have that $\alpha\in\Irr(H|\psi)$. Hence $\alpha^G\in\Irr(G)$ by Clifford's theorem. Now by \cite[Corollary 6.2]{N98}, $B_0(H)^G$ is defined and contains $\alpha^G$, and by Brauer's Third Main Theorem \cite[Theorem 6.7]{N98}, $B_0(H)^G=B_0(G)$ so we are done.
\end{proof}

\begin{lem}\label{invariant}
Let $G$ be a finite group and let $S_1\times\dots\times S_t\normal G$, where the $S_i$'s are subgroups permuted transitively by $G$. Let $S=S_1$, $S_i=S^{x_i}$ and $\alpha\in\Irr(S)$. If $\alpha$ is $\norm G S$-invariant then $\eta=\alpha^{x_1}\times\dots\times \alpha^{x_t}$ is $G$-invariant.
\end{lem}
\begin{proof}
Let $s\in S$, $x\in G$ and $H=\norm G S$. Notice that $G=\bigcup_{j=1}^tHx_j$ is a disjoint union, so $x^{-1}=hx_j$ for some $j\in \{1,\dots, t\}$ and some $h\in H$. Then $$ \eta^x(s)=\eta(s^{x^{-1}})=\eta(s^{hx_j})=\alpha^{x_j}(s^{hx_j})\prod_{i\neq j}\alpha^{x_i}(1).$$ Now $s^h\in S$ and then $\alpha^{x_j}(s^{hx_j})=\alpha(s^h)$. Since $\alpha$ is $H$-invariant we have $$\eta^x(s)=\alpha(s^h)\prod_{i\neq j}\alpha^{x_i}(1)=\alpha(s)\prod_{i\neq 1}\alpha^{x_i}(1)=\eta(s),$$
so $\eta^x(s)=\eta(s)$. 

Now $\eta^x(s^{x_j})=\eta^{xx_{j}^{-1}}(s)=\eta(s)=\eta^{x_{j}^{-1}}(s)=\eta(s^{x_j})$ applying the equality from the previous paragraph twice. In particular, for any $y\in S_i$, $\eta^x(y)=\eta(y)$.

Let $s_i\in S_i$ and notice that

$$\eta(s_1\cdots s_t)=\frac{\prod_{i=1}^t \eta(s_i)}{\prod_{i=1}^t\left(\prod_{j\neq i}\alpha_i(1)\right)}=\frac{\prod_{i=1}^t \eta(s_i)}{\alpha(1)^{t^2-t}}$$
then if $x\in G$ we have 
$$\eta^x(s_1\cdots s_t)=\eta(s_1^{x^{-1}}\cdots s_t^{x^{-1}})=\frac{\prod_{i=1}^t \eta(s_i^{x^{-1}})}{\alpha(1)^{t^2-t}}=\frac{\prod_{i=1}^t \eta(s_i)}{\alpha(1)^{t^2-t}}$$ which equals $\eta(s_1\cdots s_t)$, so we are done.
\end{proof}

\begin{lem}\label{tensor}
Let $G$ be a finite group and let $N=S_1\times\dots\times S_t\normal G$, where the $S_i$'s are subgroups transitively permuted by $G$. Let $S=S_1$. Assume that there exists $\alpha\in\Irr(B_0(S))$ with $\alpha(1)$ not divisible by $p$ and such that it extends to $\hat\alpha\in\Irr(B_0(\norm G S))$. Then the tensor induced character $\alpha^{\otimes G}$ is in the principal block of $G$.
\end{lem}
\begin{proof}
Write $H=\norm G S$. Let $\{x_1, \dots, x_t\}$ be a transversal for $H$ in $G$  such that $S^{x_i}=S_i$ and notice that $\norm G {S_i}=H^{x_i}$. Also, let $\alpha_i=\alpha^{x_i}\in\Irr(S_i)$. By Lemma \ref{invariant}, $\alpha^{x_1}\times\dots\times\alpha^{x_t}$ is $G$-invariant. Let $M=\bigcap_{i=1}^t H^{x_i}\normal G$. Since $B_0(H)$ covers $B_0(M)$ then $\hat\alpha_M\in\Irr(B_0(M))$. By \cite[Corollary 10.5]{N18}, the tensor induced character $\chi=\hat\alpha^{\otimes G}$ is an irreducible character of $G$. Also, by \cite[Lemma 10.4]{N18} we have that for $g\in M$,
 $$\chi(g)=\prod_{i=1}^t\hat\alpha^{x_i}(g).$$ Then $\chi_N=\alpha_1\times\dots\times\alpha_t\in\Irr(B_0(N))$ and $\chi_M=\prod_{i=1}^t\hat{\alpha}^{x_i}\in \Irr(M)$. Moreover, $\hat\alpha^{x_i}(1)=\alpha(1)$ is not divisible by $p$, so $\chi_M$ is in the principal block of $M$ by \cite[Lemma 3.5]{NT12}.
 
 Now let $Q\in\Syl_p(M)$. We have that $Q\cap S_i\in\Syl_p(S_i)$. Then $\cent G Q\sbs\cent G {Q\cap S_i}\sbs H^{x_i}$, and then $\cent G Q\sbs \bigcap_{i=1}^t H^{x_i}=M$.
By \cite[Lemma 3.1]{NT12} $B_0(M)^G$ is defined and is the unique block of $G$ covering $B_0(M)$. Now by Brauer's Third Main Theorem \cite[Theorem 6.7]{N98}, $B_0(M)^G=B_0(G)$ and we conclude that $\chi\in\Irr(B_0(G))$, as desired.
\end{proof}

\begin{lem}\label{uniqueblock}
Let $G$ be a finite group and assume $\cent G {\oh p G}\sbs \oh p G$. Then $G$ has a unique $p$-block.
\end{lem}
\begin{proof}
Let $P=\oh p G$. Let $b$ be the unique block of $P$. Since $\cent G P P\sbs P$, by \cite[Theorem 4.14]{N98} $b^G$ is defined. By \cite[Theorem 4.8]{N98}, $P$ is contained in every defect group of every block of $G$. Applying the second part of \cite[Theorem 4.14]{N98}, we conclude that $b^G$ is the unique block of $G$.
\end{proof}

\subsection{Necessary results on simple groups}

We shall need the following results on simple groups.

\begin{lem}\label{psl}
Let $S=\PSL_2(3^k)$ for $k\geq 3$ and let $p=3$. Then for any $S\leq T\leq\Aut(S)$ there is a character $\gamma\in\Irr(B_0(T))$ such that $\gamma_S=a\alpha$ for some $\alpha\in\Irr(B_0(S))$, $a\in\{1,2\}$ and $\alpha(1)\geq 13$ is not divisible by $3$.\end{lem}
\begin{proof}
This can be found in the proof of \cite[Proposition 3.8]{mar21}.
\end{proof}

The proof of the next three theorems is done in sections 3 and 4.

\begin{thm}\label{thm:23simple}
Let $p \in\{2,3\}$ and let $S$ be a finite nonabelian simple group of order divisible by $p$. Let $G$ be an almost simple group with socle $S$ and assume $S\not\in \{\PSL_2(3^k)\mid k\geq 2\}$ if $p=3$. 
Then there is some $\alpha\in\irr{B_0(S)}$ with $\alpha(1)>2$ that extends to $\irr{B_0(G)}$ and some $\beta\in\irr{B_0(G)}$ with $\beta(1)\neq \alpha(1)$ and $S\not\subseteq \ker{\beta}$.
\end{thm}

\begin{thm}\label{thm:2simple}
Let $p=2$ and let $G$ be almost simple with socle $S$. 
If $|\mathrm{cd}(B_0(G))|=3$, then $S$ is $\PSL_2(q)$ for some $q$.  In particular, in such a case, if $K$ is a perfect central extension of $S$,
then the Sylow $2$-subgroups of $K$ are metabelian.
\end{thm}

\begin{thm}\label{thm:3simple}
Let $p = 3$ and let $S$ be a finite nonabelian simple group of order divisible by $3$.  Let $G$ be an almost simple group with socle $S$.  If $G$ has nonabelian Sylow $3$-subgroups, then $|\mathrm{cd}(B_0(G))|>3$.  
\end{thm}

\subsection{The reduction theorem}

Next we prove our main result assuming Theorems \ref{thm:23simple}, \ref{thm:2simple}, and \ref{thm:3simple}, whose proofs will be delayed to sections 3 and 4. For this reduction theorem, we make use of the following fact: the unique integral solution for $xy=x^y$ with $x, y>1$ is $x=y=2$. Furthermore if $a, x, y$ are positive integers and $a, x>1$ then $xy<ax^y$.



\begin{thm}
Let $G$ be a finite group, $p$ a prime and assume $\cd(B_0(G))=\{1, m, n\}$. Then the Sylow $p$-subgroups of $G$ have derived length at most 3.
\end{thm}
 \begin{proof}
 We proceed by induction on $|G|$.
 
 \medskip
 {\it{Step 1: We may assume $p=2,3$, $p$ divides $m$ but not $n$ and that $\oh{p'}{G}=1$.}}
\medskip

If $p$ divides $n$ and $m$ then by \cite[Corollary 3]{isaacs-smith} we have $G/\oh{p'}{G}$ is a $p$-group. By \cite[Theorem 10.20]{N98}, $\Irr(B_0(G))=\Irr(G/\oh{p'}{G})$, so $G/\oh{p'}G$ has 3 character degrees, and then the result holds by \cite[Theorem 12.15]{Is}. If $p$ does not divide $n$ and $m$ then the result follows by the main theorem of \cite{MN21}. Hence we assume $p$ divides $m$ but not $n$. If $p>3$ then by the main result of \cite{GRSS} we have $G$ is $p$-solvable. Again, by \cite[Theorem 10.20]{N98},
$\Irr(G/\oh{p'}G)=\Irr(B_0(G))$ and then $G/\oh{p'}G$ has at most 3 degrees and we are done by \cite[Theorem 12.15]{Is}. Hence we assume $p=2,3$. Also, using $\Irr(B_0(G/\oh{p'}G))=\Irr(B_0(G))$ (by \cite[Theorem 9.9(c)]{N98}) and arguing by induction we may assume $ \oh{p'}G=1$.

\medskip
{\it{Step 2: $G$ has a nontrivial component $K$ with $|K/\zent K|$ divisible by $p$.}}
\medskip

Let $\E G$ be the layer of $G$. If $\E G=1$ then ${\bf{F}}^*(G)=\F{{G}}=\oh p G$ and then $\cent G {\oh p G}\sbs\oh p G$ by \cite[(31.13)]{aschbacher}. Then $G$ has a unique block by Lemma \ref{uniqueblock}, and we are done by \cite[Theorem 12.15]{Is}. We may assume that $\E G>1$. Let $K$ be a component of $G$ and let $N$ be the normal closure of $K$ in $G$. By \cite[(6.5.3)]{kurzweil}, $\prod_{g\in G}K^g\leq G$, so $N=\prod_{g\in G}K^g$. If $|K|$ is not divisible by $p$ then $|N|$ is not divisible by $p$, contradicting $\oh{p'}G=1$ from Step 1, so $K$ has order divisible by $p$. By \cite[(33.12)]{aschbacher} we have that $K/\zent{K}$ is also divisible by $p$. 

\medskip
{\it{Step 3: If $N$ is the normal closure of $K$ in $G$ then $N/\zent N$ is a direct product of isomorphic copies of $K/\zent K$ and $\cd(B_0(G/\zent N))=\{1, m, n\}$.}}
\medskip

By \cite[(6.5.3)]{kurzweil}, $N$ is a central product of $G$-conjugates of $K$. Write $\bar{G}=G/\zent N$. Then by \cite[(1.6.7)]{kurzweil}, $N/\zent{N}=\bar{N}=\prod \bar{K^g}=\bar{K}_1\times\dots\times\bar{K}_t$ is a direct product of non-abelian simple groups $\bar{K_i}\cong K/\zent K$, and the $\bar{K}_i$'s are transitively permuted by $\bar{G}$. We write $\bar{K_1}=\bar{K}$. If $|\cd(B_0(\bar{G}))|<3$ then by the main result of \cite{mar21} we have $\bar{G}$ is $p$-solvable, which is impossible, so $\cd(B_0(\bar{G}))=\{1, n, m\}$.  

\medskip
{\it{Step 4: If $p=3$ then $\bar{K}\ncong \PSL_2(9)$}}
\medskip

Let $\bar{H}/\bar{C}=\norm{\bar G}{\bar K}/\cent{\bar G}{\bar K}$. There is a natural isomorphism $\bar{K}\cong\bar{K}\bar{C}/\bar{C}\sbs\bar{H}/\bar{C}$ and we view $\bar{H}/\bar{K}\sbs\Aut(\bar{K})$. If $\bar{K}\cong\PSL_2(9)$ and $p=3$ then it is easy to check in \cite{GAP} that for any $\bar K\leq\bar T\leq \Aut(\bar K)$ there are two characters $\alpha,\beta\in\Irr(B_0(\bar{T}))$ that do not contain $\bar{K}$ in their kernel with $\alpha(1)\neq\beta(1)$ and both $\alpha(1)$ and $\beta(1)$ are not divisible by $p$. If $\bar{T}=\bar{H}/\bar{C}$, then by Lemma \ref{induction} we have that $\alpha^{\bar{G}},\beta^{\bar{G}}\in\Irr(B_0(\bar G))$ and $\alpha^{\bar{G}}(1)=t\alpha(1)$, $\beta^{\bar{G}}(1)=t\beta(1)$, (both $p'$ or both divisible by $p$, depending on $t$). This is a contradiction, so $\bar K$ can not be isomorphic to $\PSL_2(9)$.

\medskip

{\it{Step 5: If $p=3$ and $\bar K\cong\PSL_2(3^k)$ for $k\geq 3$ then we may assume $\bar{K}\normal\bar{G}$.}}
\medskip

Let $\alpha,\gamma$ and $a\in\{1,2\}$ be as in Lemma \ref{psl}. Since $\bar{K}$ has abelian Sylow 3-subgroups, we know that there exists $\beta\in\Irr(B_0(\bar{H}))$ with $\beta(1)\neq\gamma(1)$ and $\bar K\not\subseteq\ker\beta$ by \cite[Theorem 2.2]{mar21}. Then we can argue as before with Lemma \ref{induction}, and we get that $\gamma^{\bar G},\beta^{\bar G}\in\Irr(B_0(\bar{G}))$, so $t\gamma(1), t\beta(1)\in\cd(B_0(\bar{G}))$. Then necessarily $n=\gamma^{\bar G}(1)=t\gamma(1)=ta\alpha(1)$, and $m=t\beta(1)$. Also by \cite[Lemma 4.4]{mar21} there exists some $1\leq e\leq a^t\leq 2^t$ not divisible by 3 such that $n=e\alpha(1)^t$. Thus $ta\alpha(1)=e\alpha(1)^t$. It is easy to see that the above equality has no integral solutions verifying $\alpha(1)\geq 13$ and $t>1$. 

\medskip
{\it{Step 6: In the remaining cases, we may also assume $\bar{K}\normal \bar{G}$.}}
\medskip

Assume by way of contradiction that $t>1$. Let $\alpha\in\Irr(B_0(\bar{K}))$ with $\alpha(1)>2$, and $\beta\in\Irr(B_0(\bar{H}/\bar{C}))$ be as in Theorem \ref{thm:23simple}. Let $\hat\alpha\in\Irr(B_0(\bar{H}/\bar{C}))$ be an extension of $\alpha$. By Lemma \ref{induction}, $\hat\alpha^{\bar{G}},\hat\beta^{\bar{G}}\in\Irr(B_0(\bar{G}))$. Hence $t\alpha(1), t\beta(1)\in\cd(B_0(G))$. Since $\alpha(1)\neq\beta(1)$ we have that $t$ is not divisible by $p$ (because $n$ is not divisible by $p$). 

If $\alpha(1)$ is not divisible by $p$, then by Lemma \ref{tensor}, $\hat\alpha^{\otimes \bar{G}}\in\Irr(B_0(\bar{G}))$, so $\hat\alpha^{\otimes\bar{G}}(1)=\alpha(1)^t\in\cd(B_0(G))$. Since $\beta(1)$ is necessarily divisible by $p$, this forces $\alpha(1)^t=t\alpha(1)=n$ and if $t>1$ the only possibilty is $\alpha(1)=2=t$, a contradiction with $\alpha(1)>2$.

We are left with the case that $p$ divides $\alpha(1)$, so $\beta(1)$ is $p'$. By \cite[Corollary 10.5]{N18} $\chi=\hat\alpha^{\otimes\bar{G}}\in\Irr(\bar{G})$ extends $\eta=\alpha^{x_1}\times\dots\times\alpha^{x_t}\in\Irr(B_0(\bar{N}))$, where $\{x_1,\dots, x_t\}$ is a transversal of $\bar{H}$ in $\bar{G}$ (recall that $\eta$ is $G$-invariant by Lemma \ref{invariant}). By \cite[Theorem 9.4]{N98} there exists some $\psi\in\Irr(B_0(\bar{G}))$ over $\eta$. By Gallagher's theorem, $\psi=\mu\chi$ for some $\mu\in\Irr(\bar{G}/\bar{N})$. Since $p$ divides $\alpha(1)$ then $p$ divides $\psi(1)$ and $\psi(1)=m$. Then $\psi(1)=\mu(1)\chi(1)=\mu(1)\alpha(1)^t=m=t\alpha(1)$, but this forces $\mu(1)=1$ and $\alpha(1)=2=t$, which contradicts the fact that $\alpha(1)>2$. This proves the claim.

\medskip
{\it{Step 7: $G/\cent G K$ is almost simple with socle $K\cent G K/\cent G K\cong K/\zent K$.}}
\medskip

By Steps 5 and 6 any component $K$ of $G$ is normal in $G$. Write $C=\cent G K$ and $Z=\zent{K}$. Since $K/Z \cong KC/C\normal G/C$, it suffices to prove that $\cent {G/Z}{K/Z}=C/Z$. Let $xZ\in\cent{G/Z}{K/Z}$. We have that for all $k\in K$ there exists a unique $z_k\in Z$ such that $k^x=kz_k$. Notice that the map $k\mapsto z_k$ is a homomorphism $K\rightarrow Z$. Since $K$ is perfect and $Z$ is abelian, $z_k=1$ for all $k\in K$, so $x\in C$. The reverse inclusion is immediate. 

\medskip
{\it{Final Step.}}
\medskip

 If $|\cd(B_0(G/C))|>3$ then we are done because $\Irr(B_0(G/C))\sbs\Irr(B_0(G))$. Hence we assume $\cd(B_0(G/C))=\{1, m, n\}$ (if $|\cd(B_0(G/C))|<3$ then $G/C$ is $p$-solvable by the main result of \cite{mar21}). Since $m$ is divisible by $p$, by the main result of \cite{kessarmalle} the Sylow $p$-subgroups of $G/C$ are nonabelian. Then by Theorems \ref{thm:2simple} and \ref{thm:3simple} we have $p=2$ and that the Sylow 2-subgroups of $K$ are metabelian. By Theorem \ref{thm:23simple} there is some nontrivial $\eta\in\Irr(B_0(K))$ that extends to $\hat\eta\in\Irr(B_0(G))$. By Lemma \ref{gallagher-rizo}, $\cd(B_0(G/K))\sbs\{1, m/n\}$ and since $m/n\not\in\cd(B_0(G))$ we have $\cd(B_0(G/K))=\{1\}$. In particular, if $P\in\Syl_2(G)$ then $PK/K$ is abelian by the main result of \cite{MN21}, and $P\cap K$ is metabelian, so $\dl(P)\leq 3$.
 \end{proof}
 
 \section{Groups of Lie type}\label{sec: Mandi}
 
 The primary goal of this section is to prove that Theorems \ref{thm:23simple}, \ref{thm:2simple}, and \ref{thm:3simple} hold for simple groups of Lie type.  However, we begin by recording these statements for the sporadic groups and certain ``small" groups of Lie type.

\begin{pro}\label{prop:sporadicetc}
Theorems \ref{thm:23simple}, \ref{thm:2simple}, and \ref{thm:3simple} hold if $S$ is a sporadic simple group, $\tw{2}\type{F}_4(2)'$, 
 or a group of Lie type with exceptional Schur multiplier.  
\end{pro}
\begin{proof}

This can be seen using GAP and the GAP Character Table Library \cite{GAP}.  For the list of groups of Lie type with exceptional Schur multipliers, see \cite[Table 6.1.3]{GLS}.  
\end{proof}

Let $q$ be a power of some prime $r$. By a simple group of Lie type defined over $\FF_q$, we mean a finite simple group that can be written $S=H/\zent{H}$, where $H=\bH^F$ is the set of fixed points of a simple, simply connected algebraic group $\bH$ under a Frobenius morphism $F$ endowing $\bH$ with an $\FF_q$-rational structure. 

Let $\wt{S}$ be the group of inner-diagonal automorphisms of $S$.  Then $\Aut(S)=\wt{S}\rtimes D$, where $D$ is an appropriate group of graph and field automorphisms.  (See \cite[Section 2.5]{GLS}.) Further, we may let $\bH\hookrightarrow\wt{\bH}$ be a regular embedding as in \cite[(5.1)]{CE04} and $\wt{H}=\wt{\bH}^F$, such that $\zent{\wt{\bH}}$ is connected, $H\lhd \wt{H}$, and  $\wt{S}=\wt{H}/\zent{\wt{H}}$. 

It is useful to note that the characters of $\wt{H}$ are partitioned into so-called Lusztig series indexed by semisimple characters $s$ of an appropriate dual group $\wt{H}^\ast$. When $s=1$, the characters in this series are called \emph{unipotent} characters, and we will often work with these. 

\subsection{Some Notes on Unipotent Characters and Classical Groups}
Keeping the notation from before, unipotent characters of $\wt{H}$ are irreducible on restriction to $H$ and trivial on $\zent{\wt{H}}$.  So, we may consider unipotent characters as characters of $S, \wt{S}, H$, or $\wt{H}$.  Often, unipotent characters also extend to $\Aut(S)$.  (See \cite[Theorems 2.4 and 2.5]{Malle08}.)  One such unipotent character, the Steinberg character, has degree $|\wt{H}|_r$. For $X\in\{S, \wt{S}, H, \wt{H}\}$, we will denote the Steinberg character of $X$ by $\mathrm{St}_X$. 

The explicit list of unipotent characters and their degrees can be found in \cite[Section 13.9]{carter} for groups of exceptional type. Hence in this section, we primarily focus on the case of classical groups. 
The following will be useful for finding a rough lower bound on the number of character degrees in the principal block in this case. Throughout, we will use the notation $\PSL_n^\epsilon(q)$,  $\epsilon\in\{\pm1\}$, to denote $\PSL_n(q)$ for $\epsilon=1$ and $\PSU_n(q)$ for $\epsilon=-1$, and these are the simple groups of type $\type{A}_{n-1}(q)$ and $\tw{2}\type{A}_{n-1}(q)$, respectively.  Similarly, $\POmega_{2n}^\epsilon(q)$ will denote the simple group of type $\type{D}_n(q)$ for $\epsilon=1$ and type $\tw{2}\type{D}_n(q)$ for $\epsilon=-1$.  For completeness, we recall that $\POmega_{2n+1}(q)$ denotes the simple group of type $\type{B}_n(q)$.

\begin{pro}\label{prop:boundunipdegs}
Let $q$ be a power of some prime $r$. Let $S=\PSL^\epsilon_n(q)$ with $n\geq 2$, $\PSp_{2n}(q)$ with $n\geq 2$, $\POmega_{2n+1}(q)$ with $n\geq 3$, or $\POmega^\epsilon_{2n}(q)$ with $n\geq 4$.  Then there are more than $\frac{n+1}{2}$ unipotent characters of $\wt{S}$ (and of $S$) with distinct degrees that extend to $\Aut(S)$.
\end{pro}
\begin{proof}

The unipotent characters of $S$ or of $\wt{S}$  are in bijection with partitions of $n$ in the case of $\PSL_n^\epsilon(q)$ and in bijection with certain ``symbols" in the other cases.  (See, e.g. \cite[Section 13.8]{carter} for details.)  

If $S=\PSL_n^\epsilon(q)$, consider the characters corresponding to the partitions $(n-j, j)$ for $1\leq j\leq \lfloor \frac{n-1}{2}\rfloor$. If $S=\PSp_{2n}(q)$ or $\POmega_{2n+1}(q)$, consider the symbols ${j, n-j+1}\choose{0}$ for $1\leq j\leq \lfloor \frac{n-1}{2}\rfloor$.  In the case of $\POmega_{2n}(q)$ and $\POmega_{2n}^-(q)$, consider the symbols ${j}\choose {n-j}$ and ${j, n-j}\choose{\emptyset}$, respectively, for $1\leq j\leq \lfloor \frac{n-1}{2}\rfloor$. Using \cite[(21), (22)]{Olsson86}, the $r$-part of the degree of the corresponding unipotent characters is $q^{j}$, except in the case $r=2$ and $S=\PSp_{2n}(q)$ or $\POmega_{2n+1}(q)$, in which case it is $q^{j}/2$.  Taking into account the Steinberg and trivial characters, this shows that there are strictly more than $\frac{n-1}{2}+1$ distinct unipotent character degrees.  Further, using \cite[Theorems 2.4 and 2.5]{Malle08}, we see that each of these characters extends to $\Aut(S)$.  (Note that the exception of $\PSp_4(q)$ when $r=2$ and $j=1$ does not actually occur in our list here, as $\lfloor\frac{n-1}{2}\rfloor=0<1$ in this case.)
\end{proof}

\begin{lem}\label{lem:unipext}
Let $S$ be a simple group of Lie type and let $\wt{\chi}$ be a unipotent character of $\wt{S}$ that lies in $B_0(\wt{S})$ and extends to $\Aut(S)$.    Then $B_0(\Aut(S))$ contains an extension of $\wt{\chi}$.
\end{lem}
\begin{proof}
Recall that $\Aut(S)=\wt{S}\rtimes D$.  First, suppose that $S\neq \type{D}_4(q)$, so that the group $D$ generated by graph and field automorphisms is abelian. (See, e.g. \cite[Theorem 2.5.12]{GLS}.) Since $\wt\chi$ extends to $\Aut(S)$, we have by Gallagher's theorem that every member of $\Irr(\Aut(S))$ above $\wt{\chi}$ is an extension since $D$ is abelian.  In particular, using \cite[Theorem 9.4]{N98}, there exists an extension of $\wt{\chi}$ in $B_0(\Aut(S))$, as stated.

Now suppose for the remainder of the proof that $S=\type{D}_4(q)=\POmega_8^+(q)$. In this case, $D\cong\sym_3\times D_1$, where $D_1$ is a cyclic group of field automorphisms and $\sym_3$ is the symmetric group on 3 letters, generated by graph automorphisms.  Write $X$ for the group such that $X/\wt{S}\cong \sym_3$ and let $Y\lhd X$ such that $Y/\wt{S}$ is cyclic of size 3.   Since $\wt\chi$ extends to $\Aut(S)$, we have $\wt\chi$ extends to three characters $\hat\chi, \beta\hat\chi, \beta^{-1}\hat\chi$ of $Y$, where $\beta$ and $\beta^{-1}$ are the characters of $\Irr(Y/\wt{S})$ of order 3, using Gallagher's theorem.  Now, since $Y/\wt S$ is abelian, any block of $Y$ lying above $B_0(\wt{S})$ is of the form $B_0(Y)\otimes \gamma$ for some $\gamma\in\Irr(Y/\wt S)$.  In particular,  $B_0(Y)$ either contains all three of $\hat\chi, \beta\hat\chi, \beta^{-1}\hat\chi$ or exactly one of these characters. Say $\hat\chi\in B_0(Y)$.
 
Note that for every $\sigma\in\mathrm{Gal}(\Q(e^{2\pi i/|Y|})/\Q)$, we have $\hat\chi^\sigma$ also lies in $B_0(Y)$.  Further, $\wt\chi$ is rational-valued (see \cite[Corollary 1.12]{lusztig02}).    Then since $\chi^\sigma=\chi$, we have $\hat\chi^\sigma\in\{\hat\chi, \beta\hat\chi, \beta^{-1}\hat\chi\}$. So, if $B_0(Y)$ contains just one of the three characters, we know $\hat\chi^\sigma=\hat\chi$.  Since $\beta$ is not rational-valued,  $\hat\chi$ is then the unique character of $Y$ above $\wt\chi$  that is rational-valued.  Let $\alpha\in X$. Then since $\wt\chi$ is $\alpha$-invariant, we see $\hat\chi^\alpha\in\{\hat\chi, \beta\hat\chi, \beta^{-1}\hat\chi\}$.  But note that $\hat\chi^\alpha$ is also rational-valued, and hence $\hat\chi^\alpha=\hat\chi$.  Then $\hat\chi$ must be the extension of $\chi$ to $Y$ that extends all the way to $\Aut(S)$. On the other hand, if $B_0(Y)$ contains all three of $\{\hat\chi, \beta\hat\chi, \beta^{-1}\hat\chi\}$, we may  assume without loss that $\hat\chi$ is the one that extends to $\Aut(S)$.  

In either case, since $D/Y$ is abelian, we know every character of $\Aut(S)$ above $\hat\chi$ is an extension.  In particular,  there is an extension of $\hat\chi$ (and hence of $\wt\chi$) to $\Aut(S)$ that lies in $B_0(\Aut(S))$, again using \cite[Theorem 9.4]{N98}. 
\end{proof}

For $p$ an odd prime and $q$ an integer not divisible by $p$, we write $d_p(q)$ for the order of $q$ modulo $p$.  In what follows, if $p$ is an odd prime not dividing $q$, we will write $e$ to denote $d_p(\epsilon q)$ in the case $S=\PSL_n^\epsilon(q)$ and $d_p(q^2)$ in the case $S=\PSp_{2n}(q)$, $\POmega_{2n+1}(q)$, or $\POmega^\epsilon_{2n}(q)$.

\begin{cor}\label{cor:boundunipdegsB0}
Let $q$ be a power of some prime. Let $S$ be a simple group such that $S=\PSL^\epsilon_n(q)$ with $n\geq 2$, $\PSp_{2n}(q)$ with $n\geq 2$, $\POmega_{2n+1}(q)$ with $n\geq 3$, or $\POmega^\epsilon_{2n}(q)$ with $n\geq 4$.  Let $G$ be almost simple with socle $S$ and let $B_0(G)$ denote the principal $p$-block of $G$ for some prime $p$. 

\begin{enumerate}
\item If $p\mid q$, let $\mathcal{N}:=\frac{n-1}{2}$.  
\item If $p=2$ with $q$ odd or $p$ is odd and $e=1$, then let $\mathcal{N}:= \frac{n+1}{2}$. 
\item If $p$ is odd, $S=\PSL_n^\epsilon(q)$, $n\geq 4$, and $e=2$, then let $\mathcal{N}:=\frac{n+7}{4}$.  
\end{enumerate}
Then $B_0(G)$ contains more than $\mathcal{N}$ characters with distinct degrees obtained as extensions of unipotent characters of $S$.
\end{cor}
\begin{proof}
If $p\mid q$, then $B_0(\wt{S})$ contains all unipotent characters of $\wt{S}$ aside from the Steinberg character.  (This holds for $B_0(S)$ by a well-known result of Dagger and Humphreys, see \cite[Theorems 1.18 and 3.3]{Cab18}.  For  $B_0(\wt{S})$, \cite[Remark 6.19]{CE04} and the proof of \cite[Theorem 6.18]{CE04} again yield that every block of $\wt{G}$ is of maximal defect or defect zero.  Using Brauer's first main theorem and the fact that $\cent{\wt{G}}{U}=\zent{\wt{G}} \zent{U}$ for $U$ a Sylow-$p$-subgroup of $G$, see \cite[Remark 6.19]{CE04}, it follows that the blocks of maximal defect of $\wt{G}$ are indexed by the elements of $\Irr(\zent{\wt{G}})$.  Then using Brauer's third main theorem, $\Irr(B_0(\wt{G}))$ consists of those $\chi\in\Irr(\wt{G}|1_{\zent{\wt{G}}})$ not of defect zero.) 

If $p=2\nmid q$, then $B_0(\wt{S})$ contains all unipotent characters, using \cite[Theorem 21.14]{CE04}.   
 If $p\nmid q$ is odd, we may use the theory of $e$-cores and $e$-cocores in \cite{FS82, FS89} to see that if $e=1$, then $B_0(\wt{S})$ 
 contains the Steinberg character along with all of the unipotent characters constructed in Proposition \ref{prop:boundunipdegs}, except possibly in the case of type $\type{B}_n$ or $ \type{C}_n$ when $p\mid (q+1)$. In the latter case, either the character corresponding to the symbol ${j, n-j+1}\choose{0}$ in Proposition \ref{prop:boundunipdegs} or the character corresponding to the symbol ${0, j}\choose {n-j+1}$ lies in $B_0(\wt{S})$, and the latter still satisfies that $\chi(1)_r=q^j$ and $\chi$ extends to $\Aut(S)$.
In the case of (3),  using \cite{FS82}, we see $B_0(\wt{S})$ contains at least the unipotent characters constructed in  Proposition \ref{prop:boundunipdegs} such that $j$ is even, along with the trivial and Steinberg character.  But, $B_0(\wt{S})$ also contains the unipotent character corresponding to the partition $(n-2,1,1)$, which has $q^{3}$ as the $r$-part of the degree, which is different than that of the previously stated characters for $n\geq 4$. This gives more than $\frac{n-1}{4}+2=\frac{n+7}{4}$ unipotent characters in $B_0(\wt{S})$ that extend to $\Aut(S)$ if $n\geq 4$.

Hence in each case, $B_0(\wt{S})$ contains more than $\mathcal{N}$ characters with distinct degree that extend to $\Aut(S)$, and hence to $B_0(\Aut(S))$ by Lemma \ref{lem:unipext}. Then this yields $\mathcal{N}$ characters of distinct degree in $B_0(G)$ by restriction.
\end{proof}

We remark that when $p\in\{2,3\}$, we have either (i) or (ii) of Corollary \ref{cor:boundunipdegsB0} applies in the symplectic and orthogonal cases.

\subsection{Non-Defining Characteristic}

\begin{lem}\label{lem:23Lietypenondef}
Let $p\in\{2,3\}$ and let $S$ be a simple group of Lie type defined over $\FF_q$, where $p\nmid q$, such that $S$ is not isomorphic to a group in Proposition \ref{prop:sporadicetc}.  Then Theorem \ref{thm:23simple} holds for $S$.
\end{lem}
\begin{proof}
Let $S=H/\zent{H}$  and $\wt{S}=\wt{H}/\zent{\wt{H}}$ as above. Recall that the Steinberg character $\mathrm{St}_{\wt S}$ of $\wt S$ extends to $\Aut(S)$. Further, $\mathrm{St}_{\wt S}$ lies in the principal block, since the orders of $q$ modulo $3$ and modulo $4$ are in $\{1,2\}$, and hence the same argument as in \cite[Lemma 3.6]{RSV21} yield that $B_0(\wt{S})$ is the only block of $\wt{S}$ containing unipotent characters of degree relatively prime to $p$. So, let $\alpha=\mathrm{St}_S$ be the Steinberg character of $S$, i.e. the restriction of $\mathrm{St}_{\wt{S}}$ to $S$. Then $\alpha$ extends to $B_0(G)$ by Lemma \ref{lem:unipext}.

Now, $\alpha$ is the only irreducible character of $S$ whose degree is a power of $q$.  Hence, it suffices to know that there is some member of $\irr{B_0(S)}$ that is not trivial or the Steinberg character. This follows, for example, by \cite[Theorem C]{RSV20}.
\end{proof}

\begin{pro}\label{prop:2lietypenondef}
Theorem \ref{thm:2simple} holds if $S$ is a group of Lie type defined over $\FF_q$ with $2\nmid q$.
\end{pro}
\begin{proof}
Assume that $S$ is not isomorphic to $\PSL_2(q)$ nor one of the groups in Proposition \ref{prop:sporadicetc}, and let $S\lhd G\leq \Aut(S)$. Recall that $B_0(G)$ contains $1_G$ and an extension of $\mathrm{St}_S$, so it suffices to show that there are at least 2 additional character degrees found in $B_0(G)$.


First, assume that $S$ is the Ree group $\tw{2}\type{G}_2(3^f)$. Then there are two odd degrees other than those of $1_S, \mathrm{St}_S$ with multiplicity one, which forces the corresponding characters to be rational-valued and hence lie in $B_0(S)$, using \cite[Lemma 3.1]{NST}.  Further, since $\Out(S)$ is cyclic, we see then that these characters extend to $B_0(G)$, and we are done. 

Hence we may assume that $S$ is not $\tw{2}\type{G}_2(3^f)$, so the Sylow $2$-subgroups of $S$ are nonabelian.  (See \cite{walter}.)  Then by the recently-proven principal block case of Brauer's height zero conjecture \cite{MN21}, there is a member of $\Irr(B_0(G))$ with even degree.  On the other hand, the proof of \cite[Proposition 3.4]{RSV20} yields an odd-degree unipotent character of $B_0(S)\setminus\{1_S, \mathrm{St}_S\}$ unless $S=\PSL_3^\epsilon(q)$.   By \cite[Theorems 2.4 and 2.5]{Malle08} and arguing as in Lemma \ref{lem:unipext}, these unipotent characters extend to $B_0(G)$, except possibly if $S=\type{D}_n(q)$ or $\type{G}_2(3^f)$ and $G$ contains an automorphism of the form $\tau\varphi$ where $\tau$ is a nontrivial graph automorphism and $\varphi$ is another (possibly trivial) automomorphism. If $S=\type{D}_n(q)$ with $n\geq 5$, we are done by Corollary \ref{cor:boundunipdegsB0}.

If $S=\type{G}_2(q)$ with $q=3^f$, then the two odd-degree unipotent characters other than $1_S, \mathrm{St}_S$ have degree $\frac{1}{3}q(q^4+q^2+1)$ and fuse under the graph automorphism of order 2 but are invariant under the cyclic group of field automorphisms.  Hence in $B_0(G)$, the characters lying above these  have degree $\frac{2}{3}q(q^4+q^2+1)$ or $\frac{1}{3}q(q^4+q^2+1)$.  But using \cite{hiss-shamash}, we see  $B_0(S)$ also contains unipotent characters whose degree does not divide $\frac{2q}{3}(q^4+q^2+1)$, so $B_0(G)$ must contain at least one more distinct degree.  


In the case $\type{D}_4(q)$, there are three unipotent characters of degree $q^{10}+q^8+q^6$ and three of degree $q^6+q^4+q^2$.  The triples fuse under the graph automorphisms but are invariant under diagonal and field automorphisms, and hence the characters in $B_0(G)$ above the two triples cannot be the same degree.

Finally, let $S=\PSL_3^\epsilon(q)$. In this case, the two characters in the Lusztig series of $\wt{H}=\GL_3^\epsilon(q)$ corresponding to a semisimple element with eigenvalues $\{-1,-1,1\}$, which have degrees $q(q^2+\epsilon q+1)$ and $(q^2+\epsilon q+1)$, are invariant under $D$, trivial on $\zent{\wt{H}}$, and restrict irreducibly to $H$.  That is, we may view these as characters of $S$ that extend to $\wt{S}\rtimes D_1$, where $\wt{S}=\PGL_3^\epsilon(q)$ and $D_1$ is the cyclic subgroup of $D$ of index 2 consisting of field automorphisms.  These characters further lie in $B_0(S)$, using \cite[Theorem 21.14]{CE04}.  Then the characters of $B_0(G)$ above these two characters will have degrees $c_1\in\{q(q^2+q+1), 2q(q^2+q+1)\}$ and $c_2\in\{(q^2+q+1), 2(q^2+q+1)\}$, and hence are distinct.
 \end{proof}

\begin{pro}\label{prop:3lietypenondef}
Theorem \ref{thm:3simple} holds if $S$ is a group of Lie type defined over $\FF_q$ with $3\nmid q$.
\end{pro}
\begin{proof}

Assume that $S$ is not isomorphic to one of the groups in Proposition \ref{prop:sporadicetc} and let $S\lhd G\leq \Aut(S)$.  As before, it suffices to find two distinct degrees in $B_0(G)$ apart from those of the extensions of $1_S$ and $\mathrm{St}_S$.  If $S$ is an exceptional group of Lie type (including those of Suzuki, Ree, and $\tw{3}\type{D}_4(q)$ type), then \cite[Table 2]{RSV20} exhibits at least two additional unipotent characters in $B_0(\wt{S})$ with distinct degrees that extend to $\Aut(S)$, and we are done using Lemma \ref{lem:unipext}.  Then by applying Corollary \ref{cor:boundunipdegsB0}, we may assume that $S$ is one of the classical groups listed there with $n\leq 4$. But using \cite[Tables 4-6]{RSV20}, we are again done unless $S=\PSL_4^\epsilon(q)$ with $\epsilon\in\{\pm1\}$ such that $\epsilon q\equiv 1\pmod 3$; $S=\PSL_2(q)$; $S=\PSL_3^\epsilon(q)$, $S=\type{B}_2(2^a)$, or $S=\type{D}_4(q)$.



If $S=\type{B}_2(2^a)$, then similar to the proof of \cite[Proposition 3.9]{RSV20}, there exist unipotent characters of degree $\frac{q}{2}(q^2+1)$ and $\frac{q}{2}(q+\epsilon)^2$, where $\epsilon\in\{\pm1\}$ is such that $p\mid (q-\epsilon)$,  in $B_0(S)$ and $B_0(\wt{S})$, and the latter extends to $\Aut(S)$.  Then we are done using Lemma \ref{lem:unipext}. Similarly, the unipotent characters listed there for $\type{D}_4(q)$ work here, as they again extend to $\Aut(S)$.

Now, consider $S=\PSL_4^\epsilon(q)$ with $\epsilon\in\{\pm1\}$ and such that $\epsilon q\equiv 1\pmod 3$.  Then  there is a unique unipotent block of $S$ (and of $\wt{S}$), namely the principal block.  Since there are five unipotent characters of distinct degree, and they all extend to $\Aut(S)$, we are done in this case.  Similarly, if $S=\PSL_3^\epsilon(q)$ with $\epsilon q \equiv 1\pmod 3$, then there are 3 unipotent characters of distinct degrees, which all must lie in $B_0(S)$ and extend to $\Aut(S)$.    From here, it suffices to note that the degree $(q+\epsilon)(q^2+\epsilon q+1)$ of the third nontrivial character in $\Irr_{3'}(B_0(S))$ described in \cite[Proposition 3.11]{RSV20} does not divide the degree of any of the three unipotent characters.

Finally, assume that $S=\PSL_3^\epsilon(q)$ with $3\mid (q+\epsilon)$ or $S=\PSL_2(q)$.  Then a Sylow $3$-subgroup of $S$ is cyclic, so  our assumption that the Sylow $3$-subgroups of $G$ are nonabelian forces $3$ to divide the index $|G:S|$ and $G$ to contain some field automorphism of order $3$.  In particular, if $r$ is the prime dividing $q$, we may then write $q^2=r^{3^ba}$ for some integers $b\geq 1$ and $a\geq 1$.

Consider semisimple elements $s_1, s_2$ of $\wt{H}^\ast\cong\GL_n^\epsilon(q)$ (where $n\in\{2,3\}$ is the appropriate value) with eigenvalues $\{\delta_1, \delta_1^{-1}\}$ and $\{\delta_2, \delta_2^{-1}\}$, respectively (and with an additional eigenvalue of 1 in the case $n=3$), where $\delta_i\in\FF_{q^2}^\times$,  $\delta_1$ has order  $(r^a-1)_3$ and $\delta_2$ has order $(q^2-1)_3$.  Then for $i=1,2$, the corresponding so-called \emph{semisimple character}  $\chi_{s_i}$ of $\wt{H}=\GL_n^\epsilon(q)$ lies in $\Irr(B_0(\wt{H}))$ by \cite[Corollary 3.4]{Hiss90} and has degree relatively prime to $3$. Further, we see $\chi_{s_i}$ is trivial on the center and restricts irreducibly to $S$, using \cite[Proposition 2.6]{SFT} and \cite[Lemma 1.4]{RSV21}, since $s_i\in\SL_n^\epsilon(q)$, $\PGL_3^\epsilon(q)=\PSL_3^\epsilon(q)$ in this case, and $\zent{\GL_2(q)}$ is not divisible by $3$, so $s_iz$ cannot be conjugate to $s_i$ for any $1\neq z\in\zent{\GL_2(q)}$.  Note that for $\alpha\in D$, $\chi_{s_i}^\alpha=\chi_{s_i^\alpha}$, using \cite[Corollary 2.5]{NTT08}.  (Here since $\wt{H}^\ast\cong \wt{H}=\GL_n^\epsilon(q)$, we abuse notation and let $\alpha^\ast=\alpha$ in the notation of \cite{NTT08}.)

Now, note that the order of $\delta_2$ is strictly larger than $(r^{3^{b-1}a}-1)_3$, and hence $\delta_2^{r^{3^{b-1}a}\pm1}\neq 1$.  Recalling that semisimple classes of $\GL_n^\epsilon(q)$ are determined by their eigenvalues, we have $s_2$ cannot be conjugate to $s_2^\alpha$ for $\alpha$ a field automorphism of order $3$, and therefore $\chi_{s_2}$ is not fixed by $\alpha$.  In particular, a character in $B_0(G)$ above $\chi_{s_2}$ must have degree divisible by 3.  On the other hand, every field automorphism of order a power of $3$ sends $s_1$ to a conjugate of itself, and hence $\chi_{s_1}$ is fixed by all such field automorphisms.  Let $G_3\leq G$ where $G_3/S$ is a Sylow $3$-subgroup of $G/S$.  Note that $G_3/S$ must be generated by field automorphisms, and hence is cyclic, so we see $\chi_{s_1}$ extends to $B_0(G_3)$, and therefore there is a character in $B_0(G)$ above $\chi_{s_1}$ that is still of $3'$-degree.  Then this character cannot have the same degree as the one discussed above lying above $\chi_{s_2}$, and we are done.
\end{proof}

\subsection{Defining Characteristic}

\begin{pro}\label{prop:Lietypedefchar}
Let $p\in\{2,3\}$.  Theorems \ref{thm:23simple}, \ref{thm:2simple}, and \ref{thm:3simple} hold if $S$ is a group of Lie type defined over $\FF_q$ with $q$ a power of $p$. 
\end{pro}
\begin{proof}
We may assume that $S$ is not isomorphic to a group in Propositions \ref{prop:sporadicetc}. Recall that, from \cite[Theorem 3.3]{Cab18}, we have $\Irr(B_0(S))=\Irr(S)\setminus\{\mathrm{St}_S\}$. 

 If $S$ is one of the groups considered in Corollary  \ref{cor:boundunipdegsB0} with $n\geq 7$, we are done. 
  For the exceptional groups $\type{G}_2(q)$, $\tw{2}\type{G}_2(q)$, $\tw{3}\type{D}_4(q)$, $\type{F}_4(q)$, $\tw{2}\type{F}_4(q)$, $\type{E}_6(q)$, $\tw{2}\type{E}_6(q)$, $\type{E}_7(q)$, or $\type{E}_8(q)$ or the remaining classical groups other than $\{\PSL_2(q), \PSL_3^\epsilon(q), \PSp_4(q)\}$, we see from the explicit list of unipotent character degrees (using CHEVIE or \cite[Sections 13.8, 13.9]{carter}) and again using \cite[Theorems 2.4 and 2.5]{Malle08} that there are still at least 4 unipotent characters with distinct degrees aside from $\mathrm{St}_S$ that extend to $\Aut(S)$, and we are again done in these cases using Lemma \ref{lem:unipext}.
 We are left with the cases $S\in\{\PSL_2(q), \PSL_3^\epsilon(q), \PSp_4(q), \tw{2}\type{B}_2(q)\}$.

If $S=\PSL_3^\epsilon(q)$, there are two unipotent characters in $\Irr(B_0(S))$ (namely the trivial character and a character of degree $q(q+\epsilon)$), and as  before these extend to $B_0(G)$.  Hence we aim to show there are at least two additional distinct degrees in $\Irr(B_0(G))$. There are characters $\chi_1, \chi_2$ of degree $q+\epsilon q+1$ and $q(q+\epsilon q+1)$ that restrict irreducibly from characters $\wt{\chi}_1, \wt{\chi}_2$ of $\wt{S}:=\PGL_3^\epsilon(q)$ and have the same inertia group in $\Aut(S)$.  (These characters correspond to a Lusztig series in $\wt{H}=\GL_3^\epsilon(q)$ indexed by a semisimple element in $\wt{H}^\ast\cong\GL_3^\epsilon(q)$ with eigenvalues $\{\zeta, \zeta, \zeta^{-2}\}$ where $\zeta\in C_{q-\epsilon}\leq \FF_{q^2}^\times$ has order different from 3.)   Let $F_0$ be a generating field automorphism of $S$, so that $\Aut(S)=\wt{S}\langle \tau, F_0\rangle$  with $\tau$ a graph automorphism of order 2 in case $\epsilon=1$ and $\tau=1$ if $\epsilon=-1$.  If $1\neq\tau\in G$, then there are characters in $B_0(\wt{S}\langle F_0\rangle\cap G)$ above ${\chi}_1, {\chi}_2$ with degrees $a(q+\epsilon q+1)$ and $aq(q+\epsilon q +1)$ with $a$ some divisor of $|F_0|$, and hence we see the characters above these in $B_0(G)$ will have different degrees.  Similarly, if $G$ does not contain a non-trivial graph automorphism, then $G\wt{S}/\wt{S}$ is cyclic, and there are again characters of distinct degrees $a(q+\epsilon q+1)$ and $aq(q+\epsilon q +1)$ in $B_0(G)$ with $a$ some divisor of $|G\wt{S}/\wt{S}|$, and we are done.

If $S=\PSp_4(q)$, the unipotent characters with degree $1, \frac{1}{2}q(q-1)^2,$ and $\frac{1}{2}q(q+1)^2$  in $\Irr(B_0(S))$  extend to $B_0(G)$ as before.  The two unipotent characters of degree $\frac{1}{2}q(q^2+1)$ fuse under the exceptional graph automorphism if $p=2$ and extend to $\Aut(S)$ if $p=3$.  In either case, since the other character degrees discussed are not divisible by $q^2+1$, we see $B_0(G)$ must contain at least 4 distinct degrees.

If $S=\tw{2}\type{B}_2(q)$, note that $p=2$ and $q=2^{2m+1}$ for some $m$. Further, note that $G/S$ is cyclic, of odd order dividing $2m+1$. Then since there are two characters of $S$ of degree $2^m(2^{2m+1}-1)$, they must extend to $B_0(G)$. Let $\chi_1$ be one of these characters. There is also a character $\chi_2$ of degree divisible by $q^2+1$ in $B_0(G)$, above a character of degree $q^2+1$ of $S$. Then certainly $\chi_1(1)\neq \chi_2(1)$. Further, by considering the possible orbit sizes of $G$ acting on the characters of degree $q^2+1$ and of degree $(q-2^{m+1}+1)(q-1)$, we see that there must be a character $\chi_3$ of $B_0(G)$ of degree divisible by $(q-2^{m+1}+1)(q-1)$ and such that none of $\chi_3(1), \chi_2(1),$ or $\chi_1(1)$ are the same.  Taking into account the trivial character, this yields at least 4 character degrees in $B_0(G)$.

Finally, let $S=\PSL_2(q)$. It remains to prove Theorem \ref{thm:23simple} for  $p=2$   and Theorem \ref{thm:3simple} for $p=3$. Recall that $\Aut(S)= \wt{S}\rtimes\langle F_0\rangle$ where $\wt{S}=\PGL_2(q)$ and $F_0$ is a generating field automorphism. 

First let $S=\PSL_2(q)=\SL_2(q)$ with $q$ a power of $2$. Let $s\in\GL_2(q)$ be semisimple with eigenvalues $\{a, a^{-1}\}$, where $a\in\mathbb{F}_{q^2}^\times$ has order $3$.  Then the corresponding semisimple  character $\chi_s$ of $\GL_2(q)$ is invariant under the field automorphisms, using \cite[Corollary 2.5]{NTT08}, since the generating field automorphism $F_0$ acts by interchanging the eigenvalues of $s$ and hence $s^{F_0}$ is conjugate to $s$.  Further, since $s$ cannot be conjugate to $sz$ for any $1\neq z\in\zent{\GL_2(q)}$, we see $\chi_s$ restricts irreducibly to $\SL_2(q)=\PSL_2(q)$  using \cite[Proposition 2.6]{SFT} and \cite[Lemma 1.4]{RSV21}. So, we have $\chi_s$ extends to $\Aut(S)$ and $\alpha:=(\chi_s)_S$ has an extension in $B_0(G)$ since $G/(G\cap \wt{S})$ is cyclic.  Further, $\alpha(1)=q-\eta$ for some $\eta\in\{\pm1\}$. Then choosing $\beta$ to lie above any of the characters of $S$ of degree $q+\eta$, we are done in this case. 

Now suppose $p=3$, and we aim to prove Theorem \ref{thm:3simple} for $S=\PSL_2(q)$. If $3\nmid |G:S|$, then a Sylow $3$-subgroup of $G$ is abelian.  Hence we assume that $3\mid |G/S|$ and that the Sylow $3$-subgroups of $G$ are nonabelian. Then by the principal block version of Brauer's height-zero conjecture \cite{MN21}, we see that there must be characters of $B_0(G)$ of degree divisible by $3$.  Let $\chi_1$ be one such character. Let $q\equiv \eta\pmod 4$ with $\eta\in\{\pm1\}$.  Then the two characters of degree $\frac{1}{2}(q+\eta)$ are invariant under field automorphisms and fuse in $\wt{S}$.  Further, the character of $\wt{S}$ lying above these two characters is invariant under field automorphisms.  Hence $B_0(G)$ has a character $\chi_2$ of degree $(q+\eta)$ or $\frac{1}{2}(q+\eta)$.  The remaining nontrivial characters of $B_0(S)$ have degrees $q+\eta$ and $q-\eta$. Let $G_3$ be the subgroup of $G$ such that $G_3/S\in\mathrm{Syl}_3(G/S)$.   Since $3$ does not divide the number of characters of degree $q-\eta$ 
of $S$, we see that at least one character, say $\chi$, of degree $q-\eta$ 
must be invariant under $G_3$.  Further, note that $G_3/S$ is cyclic, so any character in $B_0(G_3)$ above $\chi$ must be an extension.  In particular, any character $\chi_3\in\Irr(B_0(G))$ above $\chi$ in $B_0(G)$ will have degree prime to $3$.  Since $\chi_3$ has degree prime to 3 and divisible by $q-\eta$, we see that $1_G, \chi_1, \chi_2,$ and $\chi_3$ have distinct degrees,  as desired. 
\end{proof}

 \section{Symmetric and alternating groups} 

The aim of this section is to give a positive answer to Question A in the cases of symmetric and alternating  groups (respectively denoted by $\fS_n$ and $\fA_n$). As already mentioned in the introduction, we can do sligthly more. Given a $p$-block $B$ of a finite group $G$ and $\chi\in\mathrm{Irr}(B)$ we let $h_p(\chi)$ denote the $p$-height of $\chi$ in $B$. Moreover, we let $\mathrm{ht}(B)=\{h_p(\chi)\ |\ \chi\in\mathrm{Irr}(B)\}.$

\begin{thm}\label{thm: Sn}
Let $p$ be a prime and let $B$ be a $p$-block of $\fS_n$ or $\fA_n$. Let $D$ be a defect group of $B$. Then $\mathrm{dl}(D)\leq |\mathrm{ht}(B)|.$
\end{thm}

As consequences of Theorem \ref{thm: Sn} above, we will show that simple alternating groups satisfy Theorems \ref{thm:23simple}, \ref{thm:2simple}, and \ref{thm:3simple}

\subsection{Notation and preliminaries}
Irreducible complex characters of the symmetric group $\mathfrak{S}_n$ are canonically labelled by partitions of $n$. In particular, given $\lambda$ a partition of $n$ (sometimes written $|\lambda|=n$) we denote by $\chi^\lambda$ the corresponding element of $\mathrm{Irr}(\fS_n)$. 
We start by recalling very briefly some useful facts on $p$-blocks of symmetric groups. 
We refer the reader to \cite{JK} or \cite{OlssonBook} for a complete description of this theory.
Let $B$ be a $p$-block of $\fS_n$. 
Then $B$ is uniquely determined by a $p$-core partition $\gamma$. More precisely two irreducible characters $\chi^\lambda$ and $\chi^\mu$ lie in the same $p$-block if and only if $\lambda$ and $\mu$ have equal $p$-core. In particular, they belong to the block $B$ if and only if they have $p$-core equal to $\gamma$. 
The integer $w=(n-|\gamma|)/p$ is called the $p$-weight of $B$. It is customary in this case to denote $B$ by $B(\gamma, w)$.
If $w=a_1+a_2p+\cdots +a_kp^{k-1}$ is the $p$-adic expansion of $w$, then a defect group $D$ of $B$ is of the following form: $$D\cong (P_p)^{\times a_1}\times (P_{p^2})^{\times a_2}\times\cdots\times (P_{p^k})^{\times a_k}, $$
where $P_m$ denotes a Sylow $p$-subgroup of $\fS_m$. In particular, we observe that $D$ is (isomorphic to) a Sylow $p$-subgroup of $\fS_{wp}$. 

\begin{lem}\label{lem: derlen}
Let $D=(P_p)^{\times a_1}\times (P_{p^2})^{\times a_2}\times\cdots\times (P_{p^k})^{\times a_k}$. Then $\mathrm{dl}(D)=k$. 
\end{lem}
\begin{proof}
Let $G'$ denote the derived subgroup of $G$. It is easy to show that $(G\times H)'=G'\times H'$. Since $P_p\cong C_p$ is a cyclic group of order $p$ and $P_{p^k}\cong (P_{p^{k-1}})\wr C_p$, the statement follows. 
\end{proof}

\subsection{Combinatorics and representations of $\fS_n$}
To prove Theorem \ref{thm: Sn} we will rely on Olsson's theory of $p$-core towers. We use this section to recall these beautiful combinatorial objects and we refer the interested reader to \cite{OlssonBook} for a more detailed account. Throughout this section, given two non-negative integers $x, y$ we denote by $[x,x+y]$ the set $\{n\in\mathbb{N}\ |\ x\leq n\leq x+y\}$.

Let $p$ be a prime number and let $\lambda$ be a partition of $n\in\mathbb{N}$. 
We denote by $c_p(\lambda)$ the $p$-core of $\lambda$ and by $q_p(\lambda)=(\lambda^{(0)},\lambda^{(1)},\ldots,\lambda^{(p-1)})$ the $p$-quotient of $\lambda$. 
With this notation in mind, we say that $\lambda^{(i)}$ is the $i$-th partition appearing in the $p$-quotient $q_p(\lambda)$ of $\lambda$.

For any sequence $(i_1,i_2,\ldots, i_k)\in [0,p-1]^{\times k}$ we inductively define
$\lambda^{(i_1,i_2,\ldots, i_k)}$ as the $i_k$-th partition appearing in the $p$-quotient of $\lambda^{(i_1,\ldots, i_{k-1})}$.
Moreover, for $i\in [0,p-1]$ we let $\lambda_{(i)}:=c_p(\lambda^{(i)})$ be the $p$-core of the $i$-th partition appearing in the $p$-quotient of $\lambda$. 
For any sequence $(i_1,i_2,\ldots, i_k)\in [0,p-1]^{\times k}$ we let
$\lambda_{(i_1,i_2,\ldots, i_k)}=c_p(\lambda^{(i_1,i_2,\ldots, i_k)}).$

We denote by $T(\lambda)$ the \textit{$p$-core tower} of $\lambda$ (see \cite[Section 6]{OlssonBook}). In particular we find convenient to think of the $p$-core tower as a sequence $T(\lambda)=(T_{j}(\lambda))_{j=0}^{\infty}$, 
where the $k$-th layer (or row) $T_k(\lambda)$ is the sequence of $p^k$ $p$-core partitions defined as follows: $T_0(\lambda)=(c_p(\lambda))$, and for $k\geq 1$, 
$$T_k(\lambda)=\big(\lambda_{(0,\ldots,0)},\ldots, \lambda_{(i_1,i_2,\ldots, i_k)}, \ldots, \lambda_{(p-1,\ldots, p-1)}\big),\ \ \text{for}\ \ 0\leq i_1,\ldots, i_k\leq p-1.$$

 We denote by $|T_k(\lambda)|$ the sum of the sizes of the $p$-cores in the $k$-th layer of $T(\lambda)$. 
We have that $|\lambda|=\sum_j|T_j(\lambda)|p^j$. Moreover, every partition $\lambda$ of $n$ is uniquely determined by its $p$-core tower. This follows by repeated applications of  \cite[Proposition 3.7]{OlssonBook}. 

In \cite{Olsson} the following fundamental result is proved. 

\begin{thm}\label{Mac}
Let $p$ be a prime number and let $\lambda$ be a partition of $n\in\mathbb{N}$. Let $B=B(\gamma, w)$ be the $p$-block of $\chi^\lambda$.
Suppose that $wp=\sum_{j=1}^ka_jp^j$ is the $p$-adic expansion of $wp$. 
Then $$h_p(\chi^\lambda)=\big(\sum_{j=1}^k |T_j(\lambda)|-a_j\big)/(p-1).$$
\end{thm}

\subsection{Proof of Theorem \ref{thm: Sn}}

Let $n$ be a natural number, let $B=B(\gamma, w)$ be a $p$-block of $\fS_n$ and let $wp=\sum_{j=1}^ka_jp^j$ be the $p$-adic expansion of $wp$, where $a_k\neq 0$. 
This notation will be kept throughout the section. 

\begin{defi}\label{def: 0}
Given a prime $p$ and an integer $a\in [1,p-1]$, we denote by $\gamma_a$ the partition $(p, 1\6a)$.  It is easy to observe that $\gamma_a$ is a $p$-core partition of $p+a$. 
\end{defi}

We are now ready to introduce the main combinatorial objects of this section. 
In order to do this we mention that the empty partition will usually be denoted by $\emptyset$ or by $(0)$ depending on our convenience. 

\begin{defi}\label{def: 1}
Let $\lambda_0, \lambda_1,\ldots, \lambda_{k-1}$ be the partitions of $n$ defined as follows. 
We let $\lambda_0$ be the partition such that $T_0(\lambda_0)=(\gamma)$ and for all $i\in [1,k]$ we let $T_i(\lambda_0)=((a_i), \emptyset, \ldots, \emptyset)$. 

For every $j\in [1,k-1]$ we recursively define $\lambda_j$ by modifying layers $k-j$ and $k-j+1$ of the $p$-core tower corresponding to $\lambda_{j-1}$.
More precisely, we
let $\lambda_j$ be the partition corresponding to the following $p$-core tower: $T_0(\lambda_j)=(\gamma)$, $T_i(\lambda_j)=T_i(\lambda_{j-1})$ for all $i\notin\{k-j, k-j+1\}$, and 
\[ T_{k-j}(\lambda_j) = \begin{cases}
	(\gamma_a, \emptyset, \ldots, \emptyset) & \mathrm{if}\ T_{k-j}(\lambda_{j-1})=((a),\emptyset, \ldots, \emptyset),\ \exists\ a\in [1,p-1]\\
	((p-1), (1),\emptyset, \ldots, \emptyset) & \mathrm{if}\ T_{k-j}(\lambda_{j-1})=(\emptyset, \ldots, \emptyset),
	\end{cases}\]
and

\[ T_{k-j+1}(\lambda_j) = \begin{cases}
((a-1),\emptyset, \ldots, \emptyset)& \mathrm{if}\ T_{k-j+1}(\lambda_{j-1})=((a),\emptyset, \ldots, \emptyset),\ \exists\ a\in [1,p-1]\\
(\gamma_{a-1}, \emptyset, \ldots, \emptyset) & \mathrm{if}\ T_{k-j+1}(\lambda_{j-1})=(\gamma_a,\emptyset, \ldots, \emptyset),\ \exists\ a\in [2,p-1]\\
	((p-1), (1),\emptyset, \ldots, \emptyset) & \mathrm{if}\ T_{k-j+1}(\lambda_{j-1})=(\gamma_1,\emptyset, \ldots, \emptyset),\\
	((p-1),\emptyset, \ldots, \emptyset) & \mathrm{if}\ T_{k-j+1}(\lambda_{j-1})=((p-1),(1),\emptyset, \ldots, \emptyset).
	\end{cases}\] 
\end{defi}

We will now verify that for every $j\in [0,k-1]$ the partition $\lambda_j$ is well-defined and labels a character lying in $B=B(\gamma, w)$. Afterwards we will show that $h_p(\lambda_j)=h_p(\lambda_i)$ if and only if $i=j$. This will allow us to conclude that $\mathrm{dl}(D)\leq |\mathrm{ht}(B)|$.

\begin{lem}\label{lem: 1}
Given $j\in [0,k-1]$ we have that $\lambda_j$ is well-defined and labels a character lying in $B=B(\gamma, w)$.
\end{lem}
\begin{proof}
We proceed by induction on $j\in [0,k-1]$. The base case is immediate. The partition $\lambda_0$ is well defined, $\gamma=T_0(\lambda_0)$ is the $p$-core of $\lambda_0$ and $\lambda_0$ is a partition of $n$ since
$$|\lambda_0|=\sum_{i=0}\6k|T_i(\lambda_0)|p\6i=\sum_{i=0}\6ka_ip\6i=n.$$ 
Let now $j\in [1,k-1]$. Let us start by showing that the definition of its $p$-core tower is consistent. In order to do this we need to show that each row $T_i(\lambda_j)$ is well-defined. If $i\notin \{k-j, k-j+1\}$ then no problems arise, as $T_i(\lambda_j)=T_i(\lambda_{j-1})$. For $i\in\{k-j, k-j+1\}$, we need to verify that 
$$T_{k-j}(\lambda_{j-1})\in\{((a),\emptyset, \ldots, \emptyset)\ |\ a\in [0,p-1]\},$$
and that 
$$T_{k-j+1}(\lambda_{j-1})\in\{((a),\emptyset, \ldots, \emptyset), (\gamma_a,\emptyset, \ldots, \emptyset), ((p-1),(1),\emptyset, \ldots, \emptyset)\ |\ a\in [1,p-1]\}.$$
Both statements are verified again by induction (on $j-1$). They are clearly satisfied by $\lambda_0$, and for $j-1\geq 1$ they remain true by direct verification (using the recursive definition of the partitions $\lambda_0, \ldots, \lambda_{k-1}$ given in Definition \ref{def: 1}). 
Now, we want to show that $\lambda_j$ is a partition of $n$ for all $j\in [0,k-1]$. The base case $j=0$ has been already verified above. Let $j\geq 1$ and let us suppose that $|\lambda_{j-1}|=n$. Then we observe that $T_i(\lambda_{j})=T_i(\lambda_{j-1})$, for all $i\in [0,k]\smallsetminus\{k-j, k-j+1\}$. 
Moreover, we have that $|T_{k-j}(\lambda_{j})|=|T_{k-j}(\lambda_{j-1})|+p$ and that $|T_{k-j+1}(\lambda_{j})|=|T_{k-j+1}(\lambda_{j-1})|-1$. It follows that 
$$|\lambda_j|=\sum_{i=0}\6k|T_i(\lambda_j)|p\6i=\sum_{i=0}\6k|T_i(\lambda_{j-1})|p\6i+p\cdot p\6{k-j}-1\cdot p\6{k-j+1}=|\lambda_{j-1}|=n.$$ 
Since $c_p(\lambda_j)=T_0(\lambda_j)=\gamma$, we deduce that $\chi\6{\lambda_j}\in \mathrm{Irr}(B)$, for all $j\in [0,k-1]$. 
\end{proof}

We conclude by showing that the irreducible characters labelled by the partitions $\lambda_0, \ldots, \lambda_{k-1}$ have distinct character degrees. In the following statement we keep the notation introduced in Definition \ref{def: 1}. 

\begin{pro}\label{prop: 1}
Let $j\in [0,k-1]$. Then $h_p(\chi\6{\lambda_j})=j$. 
\end{pro}
\begin{proof}
To simplify the notation, given a partition $\mu$ we will denote by $h_p(\mu)$ the integer $h_p(\chi\6\mu)$. We also recall that $w$ denotes the $p$-weight of the block $B(\gamma, w)$ and that $wp$ has $p$-adic expansion given by $$wp=\sum_{i=1}\6ka_ip\6i.$$

We proceed by induction on $j\in [0,k-1]$. Using Theorem \ref{Mac} we see that 
$$h_p(\lambda_0)=\big(\sum_{i=1}\6k|T_i(\lambda_0)|-a_i\big)/(p-1)=0.$$
The second equality follows from the definition of $\lambda_0$ having $|T_i(\lambda_0)|=a_i$ for all $i\geq 1$. 
Let now $j\geq 1$ and let us assume that $h_p(\lambda_{j-1})=j-1$. As already explained in the proof of Lemma \ref{lem: 1} we have that $T_i(\lambda_{j})=T_i(\lambda_{j-1})$, for all $i\in [0,k]\smallsetminus\{k-j, k-j+1\}$. 
Moreover, we have that $|T_{k-j}(\lambda_{j})|=|T_{k-j}(\lambda_{j-1})|+p$ and that $|T_{k-j+1}(\lambda_{j})|=|T_{k-j+1}(\lambda_{j-1})|-1$.
Using Theorem \ref{Mac}, it follows that 
$$(p-1)h_p(\lambda_j)=\big(\sum_{i=1}\6k|T_i(\lambda_j)|-a_i\big)=\big(\sum_{i=1}\6k|T_i(\lambda_{j-1})|-a_i\big)+p-1=(p-1)h_p(\lambda_{j-1})+p-1.$$
Using the inductive hypothesis we get $h_p(\lambda_j)=h_p(\lambda_{j-1})+1=(j-1)+1=j$. 
\end{proof}

We are now ready to prove Theorem \ref{thm: Sn} and, as a consequence, to verify Theorems \ref{thm:23simple}, \ref{thm:2simple}, and \ref{thm:3simple} for alternating groups $\fA_n$. We recall that, given a partition $\lambda$ of $n$,  the restriction $(\chi\6{\lambda})_{\fA_n}$ is irreducible whenever $\lambda$ is not equal to its conjugate partition $\lambda'$. Moreover, if $\chi\6{\lambda}\in \mathrm{Irr}(B_0(\fS_n))$ then every irreducible constituent of $(\chi\6{\lambda})_{\fA_n}$ lies in $\mathrm{Irr}(B_0(\fA_n))$.

\begin{cor}
Let $G\in\{\fS_n, \fA_n\}$ and let $B$ be a $p$-block of $G$ with defect group $D$. Then $\mathrm{dl}(D)\leq |\mathrm{ht}(B)|$.
\end{cor}
\begin{proof}
If $n\leq 4$ the statement holds by direct verification. Hence, let us assume that $n\geq 5$. 
We first deal with the case where $G=\fS_n$.
Let $B=B(\gamma, w)$ for some $p$-core $\gamma$ and some integer $w$ such that $n=|\gamma|+wp$. Let $wp=\sum_{i=1}\6ka_ip\6i$ be its $p$-adic expansion. Then by Lemma \ref{lem: derlen} we have that $\mathrm{dl}(D) = k$. From the discussion started in Definition \ref{def: 1} and ended with Proposition \ref{prop: 1} we know that the $p$-block $B$ admits $k$ irreducible characters whose $p$-heigths are pairwise distinct. It follows that $|\mathrm{ht}(B)|\geq k= \mathrm{dl}(D)$, as desired. 

Let now $G=\fA_n$. Let $B$ be a $p$-block of $\fA_n$, covered by the $p$-block $\tilde{B}$ of $\fS_n$. 
Let $\gamma$ and $w$ be such that $\tilde{B}=B(\gamma, w)$ and let $wp=\sum_{i=1}\6ka_ip\6i$ be its $p$-adic expansion. Let $\lambda_0,\lambda_1,\ldots, \lambda_{k-1}$ be the partitions of $n$ described in Definition \ref{def: 1}.   
 Let $Q$ be a defect group of $B$ and $D$ a defect group of $\tilde{B}$ chosen so that $Q\leq D$ (if $p\neq 2$ then $Q=D$).
Since $n\geq 5$ \cite[Proposition 3.5]{OlssonBook} shows that $(\lambda_i)'\neq \lambda_j$ for all $i,j\in [0,k-1]$. It follows that $\{\mathrm{h}_p((\chi\6{\lambda_i})_{\fA_n})\ |\ i\in\ [0,k-1]\}$ is a subset of size $k$ of $\mathrm{ht}(B)$. We conclude that $\mathrm{dl}(Q)\leq\mathrm{dl}(D)= k\leq |\mathrm{ht}(B)|$.
\end{proof}

\begin{cor}
Theorem \ref{thm:23simple} holds for alternating groups $\fA_n$, for $n\geq 5$. 
\end{cor}
\begin{proof}
Let $p\in\{2,3\}$ and $n$ be a natural number. 
The cases $n\leq 7$ for $p=2$ and $n\leq 26$ for $p=3$ can be easily checked in GAP. 
Let us now suppose to be in one of the remaining cases. Let $k\in \mathbb{N}$ be such that $p\6k\leq n< p\6{k+1}$. Let $B=B_0(\fS_n)$. Since $k\geq 3$, recycling the notation introduced in Definition \ref{def: 1} setting $\gamma=c_p((n))$, it follows that $\lambda_1$ and $\lambda_2$ are well-defined partitions of $n$, 
labelling irreducible characters lying in $B$. Using \cite[Proposition 3.5]{OlssonBook} we observe that 
$(\lambda_i)'\neq \lambda_i$, for all $i\in\{1,2\}$. It follows that $\alpha:=(\chi\6{\lambda_2})_{\fA_n}\in\mathrm{Irr}(B_0(\fA_n))$ and $\alpha(1)>2$ because $p\6{2}$ divides $\alpha(1)$ by Proposition \ref{prop: 1}. Similarly $\beta:=\chi\6{\lambda_1}\in\mathrm{Irr}(B)$. By Proposition \ref{prop: 1} we know that $p$ divides $\beta(1)$ and that $p\62$ does not. It follows that $\fA_n$ is not a subgroup of $\mathrm{Ker}(\beta)$ and that $\alpha(1)\neq \beta(1)$. 
\end{proof}

\begin{cor}
Theorems \ref{thm:2simple} and \ref{thm:3simple}  hold for alternating groups $\fA_n$, for $n\geq 5$. 
\end{cor}
\begin{proof}
We start with Theorem \ref{thm:2simple}.
Let $B$ be the principal $2$-block of $\fS_n$. If $5\leq n\leq 15$ the statement can be checked in GAP. Let us now assume that $n\geq 16$. Setting $\gamma=c_p((n))$ we have that the partitions $\lambda_0, \lambda_1, \lambda_2$ and $\lambda_3$ are well defined partitions of $n$ and label irreducible characters of $\fS_n$ lying in $B$. By \cite[Proposition 3.5]{OlssonBook} we know that $(\lambda_i)'\neq \lambda_i$ for all $0\leq i\leq 3$. It follows that $\{(\chi\6{\lambda_i})_{\fA_n}\ |\  0\leq i\leq 3\}$ is a subset of size $4$ of $\mathrm{Irr}(B_0(\fA_n))$. 

To conclude we now verify Theorem \ref{thm:3simple}. Clearly $n\geq 9$ as a Sylow $3$-subgroup of $\fA_n$ is nonabelian. In this case we consider $\lambda=(n)$, $\mu=(n-3,1\6{3})$ and $\eta=(n-3, 3)$. 
Clearly $(\chi\6{\lambda})_{\fA_n}, (\chi\6{\mu})_{\fA_n}, (\chi\6{\eta})_{\fA_n}\in\mathrm{Irr}(B_0(\fA_n))$ and have pairwise distinct character degrees.   
\end{proof}
\begin{rem}
We conclude by mentioning that Theorem \ref{thm: Sn} was already known to hold in the case of symmetric groups for primes $p\geq 5$ \cite{moreto}. Moreto's proof relies on the important result of Granville and Ono \cite{GO96}, showing that for every natural number $n$ there exists a $p$-core partition of size $n$, provided that $p$ is at least $5$. Our proof does not use this fact; it is uniform for all prime numbers and allows us to deduce the needed results on alternating groups. 
\end{rem}

\section{More on Question \ref{a}}\label{sec: 5}

In this section, we discuss certain cases of Question \ref{a} and related questions. We denote by ${\rm{ht}}(B)$ the set of heights of characters in a block $B$.

\subsection{Blocks of the general linear group in defining characteristic}

Using results of A. Moret\'o, we show that Question \ref{a} holds for blocks of the general linear group in defining characteristic.

\begin{lem}\label{moreto}
If $B$ is a $p$-block of positive defect of $\GL_n(q)$, where $q$ is a power of $p$, then $|{\rm{ht}}(B)|\geq n-1$.
\end{lem}
\begin{proof}
This is \cite[Lemma 3.1]{moreto}.
\end{proof}

\begin{lem}\label{hup}
The Sylow $p$-subgroups of $\GL_n(q)$, where $q$ is a power of $p$, have nilpotency class $n-1$.
\end{lem}
\begin{proof}
This is \cite[Satz III.16.3]{hup}.
\end{proof}

\begin{cor}\label{gl}
Question \ref{a} has an affirmative answer for the blocks of $\GL_n(q)$ in defining characteristic.
\end{cor}
\begin{proof}
The result follows from Lemmas \ref{moreto} and \ref{hup} using that the derived length is bounded by the nilpotency class and $|{\rm{ht}}(B)|\leq |\cd(B)|$.
\end{proof}

\begin{cor}\label{sl}
Let $N\normal \GL_n(q)$ for a power $q$ of a prime $p$, and assume $|\GL_n(q):N|$ is not divisible by $p$. Then Question \ref{a} has an affirmative answer for the $p$-blocks of $N$. In particular, Question \ref{a} holds for the blocks of $\SL_n(q)$ in the defining  characteristic. 
\end{cor}
\begin{proof}
Let $b$ be a block of $N$ and let $B$ be a block of $\GL_n(q)$ covering $b$. By \cite[Theorem 9.26]{N98}, the defect groups of $B$ are defect groups of $b$. Let  $\psi\in\Irr(b)$ and let $\chi\in\Irr(B)$ be over $\psi$ (see \cite[Theorem 9.4]{N98}). It follows from \cite[Theorem 5.12]{N18} that $\psi(1)_p=\chi(1)_p$ so we see that ${\rm{ht}}(B)={\rm{ht}}{(b)}$, so this result follows from the argument of the proof of Corollary \ref{gl}.
\end{proof}

\subsection{Checking GAP libraries}

Let $p$ be a prime and let $D$ be a $p$-group of size $p^k$. Since groups of order $p^2$ are abelian, it is easy to see that $\dl(D)\leq \frac{k+1}{2}$. Using this bound, we have checked in \cite{GAP} that Question \ref{a} holds for all sporadic groups. Furthermore, Question \ref{a} has been checked for all perfect groups and the primitive groups of degree up to $1500$ and size up to $10^6$. Whenever the bound above does not work, in most cases it suffices to check that $|\cd(B)|\leq \dl(P)$ for a Sylow $p$-subgroup $P$, since $\dl(D)\leq\dl(P)$ for any defect group $D$. Otherwise, the fact that if $D$ is a defect group of some block, then $D=\oh p {\norm G D}$ (see \cite[Corollary 4.18]{N98}) helps locating possible defect groups of a block if the defect is known.

\subsection{Related questions}

If $D$ is a $p$-group, Taketa's Theorem \cite[Theorem 5.12]{Is} states that $|\cd(D)|\geq \dl(D)$, so questions regarding the derived length of $D$ are closely related to its set of character degrees. The following conjecture was recently asked in \cite{flz}.

\begin{conj}[Feng--Liu--Zhang]\label{feng}
Let $B$ be a $p$-block of a finite group $G$ with defect group $D$. Let $p^a$ be the maximal element in $\cd(D)$ and let $b$ be the maximal height of the characters in $B$. Then $a\leq b$.
\end{conj}

Using \cite{GAP} notation, let $G=\verb+SmallGroup(729, 955)+$. Let $A=\Aut(G)$ and $Q\in\Syl_{13}(A)$. Let $D\in\Syl_{3}(\norm A Q)$ and $R=\langle D, Q \rangle\leq A$. We consider the semidirect product $H=G\rtimes R$. Then $H$ is a solvable group of size $28431=3^7\cdot 13$, with $\oh{3'}H=1$. Then $\Irr(H)=\Irr(B_0(H))$ and $\cd(H)=\cd(B_0(H))=\{1,3,13, 39\}$. However if $P\in\Syl_p(H)$ we have $\cd(P)=\{1, 3, 9\}$, so $H$ is a counterexample to Conjecture \ref{feng}. This example is based on an unpublished note of I. M. Isaacs regarding a similar question. In fact, this provides a counterexample to Conjecture \ref{feng} for all odd primes, as asked by G. Malle after G. Navarro found the counterexample  $\verb+SmallGroup(192,955)+$ for $p=2$.

In Corollaries \ref{gl} and \ref{sl}, as well as in the proof of Theorem \ref{thm: Sn}, we have used the bound $\dl(D)\leq|{\rm ht}(B)|$ taking advantage of the fact that $|{\rm ht}(B)|\leq |\cd(B)|$. It is a natural question to ask if we can improve the bound in Question \ref{a} by $\dl(D)\leq|{\rm ht}(B)|$ in general. However this bound does not hold even in solvable groups; the group $H$ constructed above is a counterexample, as its Sylow $3$-subgroups have derived length 3. We mention that for $p=2$, the group \verb+PerfectGroup(17280,1)+ is a counterexample for this bound. No solvable counterexamples have been found for $p=2$.

Finally, we would like to remark that if $G$ is solvable and $B$ is the principal block, then an affirmative answer to Question \ref{a} would follow from assuming the Isaacs--Seitz conjecture (see \cite[(8.2)]{N10}).

\end{document}